\documentclass[12pt]{article}

\usepackage[letterpaper, hmargin=1.1in, top=1in, bottom=1.2in, footskip=0.6in]{geometry}

\makeatletter\def\SetFigFont#1#2#3#4#5{\small}

\usepackage{amsmath} 
\usepackage{amssymb}
\usepackage{amsfonts}
\usepackage{latexsym}
\usepackage{amsthm}
\usepackage{amsxtra}
\usepackage{amscd}
\usepackage{url}
\usepackage{bm}

\usepackage{graphicx, color}

\numberwithin{equation}{section}
\numberwithin{figure}{section}

\theoremstyle{plain} 
\newtheorem{theorem}{Theorem}
\newtheorem*{theorem*}{Theorem}
\newtheorem{lemma}[theorem]{Lemma}
\newtheorem*{lemma*}{Lemma}

\newtheorem*{corollary*}{Corollary}
\newtheorem{proposition}[theorem]{Proposition}
\newtheorem*{proposition*}{Proposition}

\newtheorem*{definition*}{Definition}

\newtheorem*{conjecture*}{Conjecture}

\newtheorem*{example*}{Example}

\newtheorem*{remark*}{Remark}

\definecolor{darkred}{rgb}{0.9,0,0.3}
\definecolor{darkblue}{rgb}{0,0.3,0.9}

\usepackage{ifthen}
\def\comment#1{\ifthenelse{\isodd{\value{page}}}{\marginpar{\raggedright\scriptsize{\textcolor{darkred}{#1}}}}{\marginpar{\raggedleft\scriptsize{\textcolor{darkred}{#1}}}}}

\renewcommand{\leq}{\leqslant}

\renewcommand{\epsilon}{\varepsilon}

\newcommand{\be}{\begin{equation}}
\newcommand{\eeq}{\end{equation}}
\newcommand{\bea}{\begin{align}}

\begin{document}
\title{Central limit theorem near the critical temperature for the overlap in the 2-spin spherical SK model}

\author{Vu Lan Nguyen\thanks{{\tt nguyenvulan47@gmail.com}} \and Philippe Sosoe\thanks{{\tt psosoe@math.cornell.edu}, P.S.'s research is supported by NSF grant DMS-1811093.}}

\maketitle

\abstract{We prove a central limit theorem for the normalized overlap in the spherical SK model in the high temperature phase. The convergence holds almost surely with respect to the disorder variables, and the inverse temperature can approach the critical value at a polynomial rate with any exponent strictly greater than $1/3$.}

\section{Introduction}
For $N\in \mathbb{N}$, the spherical Sherrington-Kirkpatrick (SSK) Hamiltonian is given by
\begin{equation}
\label{def:H_spherical}
H_N(\sigma)= \sum_{1\leq i,j\leq N} -\frac{1}{\sqrt{ N}}g_{ij}\sigma_i\sigma_j,
\end{equation}
for $\sigma:=(\sigma_1,\cdots,\sigma_N)\in \mathbb{S}^{N-1}=\{ \sigma \in \mathbb{R}^N, |\sigma|=\sqrt{N}\}$. Here the $g_{ij}$ are independent standard Gaussian random variables. This model was introduced in \cite{KTJ} and has been extensively studied \cite{talagrand}, \cite{crisanti-sommers}, \cite{talagrand-panchenko}. It is a spin glass model whose study is somewhat simpler than that of the related SK model with Ising spins, because a number of more explicit computations are possible.

One such computation is the remarkable result by J. Baik and J.O. Lee \cite{baiklee} concerning the fluctuations of the logarithm \emph{partition function} $Z_N(\beta)$, defined by
\begin{equation}
\label{def:Z_spherical}
Z_N(\beta)=\frac{1}{|\mathbb{S}^{N-1}|}\int_{\mathbb{S}^{N-1}}e^{-\beta H_N(\sigma)}\mathrm{d}\omega(\sigma),
\end{equation}
where $\mathrm{d}\omega (\sigma)$ is the uniform measure on the sphere $\mathbb{S}^{N-1}$. 
\begin{theorem*}[Baik and Lee \cite{baiklee}]
Let
\[F_N(\beta)=\frac{1}{N}\log Z_N(\beta)\]
be the free energy of the spherical SK model at inverse temperarture $\beta$. Then we have the following, where $\rightarrow$ denotes convergence in distribution:
\begin{enumerate}
\item \begin{equation}\label{eqn: high-T}
N(F_N(\beta)-F(\beta)\rightarrow \mathcal{N}(f,\alpha),
\end{equation}
where
\[f=\frac{1}{4}\log(1-16\beta^2)-8\beta^2, \quad \alpha=-\frac{1}{2}\log(1-16\beta^2)-8\beta^2.\]
\item In the low temperature regime $\beta>1$,
\begin{equation}\label{eqn: TW}
\frac{1}{\beta-1}N^{2/3}(F_N(\beta)-F(\beta))\rightarrow TW_1,
\end{equation}
the Tracy-Widom GOE distribution.
\end{enumerate}
\end{theorem*}
For the Ising SK Hamiltonian, the analogous result to \eqref{eqn: high-T} was derived by Aizenman, Lebowitz and Ruelle in the high temperature phase, while the fluctuations of the partition function in the low temperature remain inaccessible. The best results are bounds for the order of fluctuations which fall far short of their expected size, see \cite[Theorem 6.3]{chatterjee}.

Baik and Lee's result is based on a contour integral representation for the partition function, which reduces the evaluation of $F_N(\beta)$ to a (delicate) saddle point analysis involving quantities studied in random matrix theory. In a subsequent series of work, these authors exploit similar representations to derive the thermodynamic limit of variants of the SSK model, including a model with an additional ferromagnetic (Curie-Weiss) interaction in the Hamiltonian \cite{baiklee1}, the bipartite SSK model \cite{baiklee2}, for which the computation of the first order behavior at all temperatures was open, and in a recent joint work with H. Wu, the SSK/Curie-Weiss model at the critical coupling strength \cite{baiklee3}.

In this note, we observe that a representation as a \emph{double} contour integral is available for the Laplace transform of the overlap $R_{12}$ with respect to the Gibbs measure, a key quantity in the study of spin glasses. The overlap between two replicas $\sigma_1$ and $\sigma_2$ is defined as:
\begin{equation}
\label{def:overlap_R}
R_{12}=\frac{1}{{N}}\langle\sigma_1,\sigma_2\rangle.
\end{equation}
More explicitly:
\begin{equation}
\label{computation:Laplace_R}
\langle e^{tR_{12}} \rangle_{\beta,N}= \frac{1}{Z_N^2} \frac{1}{|\mathbb{S}^{N-1}|^2}\int_{(\mathbb{S}^{N-1})^2} \exp \bigg (-\beta(H_N(\sigma_1)+H_N(\sigma_2))+tR_{12}\bigg )d\omega(\sigma_1) d\omega(\sigma_2).
\end{equation}
As a consequence, we show that the overlap, properly rescaled, has Gaussian fluctuations in the high temperature phase, including at inverse temperatures close to the critical point:
\begin{theorem}\label{thm: clt}
Let $\beta=\beta_N=1-cN^{-1/3+\tau}$ with $0<\tau<1/3$. Then there is a $c(\tau)>0$ such that
\[\langle e^{tR_{12}}\rangle_{\beta,N} = e^{t^2}+O(N^{-c(\tau)})\]
with probability at least $1-N^{-K}$, for any $K$, as $N\rightarrow \infty$.

In particular, for almost every realization of the disorder variables $g_{ij}$, the overlap $R_{12}$ converges in distribution to a Gaussian random variable with mean 0 and variance 2.
\end{theorem}
\paragraph{Remarks}
\begin{enumerate}
\item Note that the condition $\tau>0$ is optimal, since for $\tau=0$, one does not expect the overlap to be asymptotically Gaussian. See \cite[Theorem 11.7.2]{talagrand-book}, where this is proved for the SK model with Ising spins.
\item The input we need concerning the matrix of random couplings $g_{ij}$ is the regularity of the limiting spectral distribution, and a strong version of the local law. It follows that the disorder can be replaced by a large class of distributions with minimal changes. We do not comment on this further. This point is explained in detail in \cite[Section 2]{baiklee}, for example.
\end{enumerate}

\subsection{Discussion}
For the SK model with Ising ($\pm 1$) spins, M. Talagrand \cite[Theorem 11.7.1]{talagrand-book} proved that the overlaps for several replicas are asymptotically jointly Gaussian under the ``annealed'' measure $\mathbf{E}\langle \cdot \rangle_{\beta,N}$ whenever $N^{1/3}(1-\beta_N^2)\rightarrow \infty$. Here $\mathbf{E}$ denotes expectation over the disorder variables. He also showed that the limiting distribution is not Gaussian if $N^{1/3}(1-\beta_N^2)\rightarrow c>0$. The proof used moment computations using the cavity method.

It is natural to ask whether the representation \eqref{eqn: double-contour}, can be used to study the distribution of the overlap in the low temperature phase. It is known \cite{talagrand-panchenko} that in this phase, the overlap concentrates about the values $\pm q$, where $q=1-1/\beta$. The behavior of the partition function was the subject of \cite{baiklee} and the limiting distribution was remarkably found to be given by the Tracy-Widom distribution when $\beta>1$. One can then ask what the typical size of fluctuations of $R_{12}^2$ are about the limiting value. This is the subject of forthcoming work.

As indicated in \cite{baiklee} and subsequent papers, it would be of great interest to investigate the transition region $|\beta-1|\ll N^{-1/3}$. The main difficulty here, for both the overlap and the partition function, is the presence of a branch point very near the random saddle point, with the distance between this singularity and the saddle point being far smaller than the order of the fluctuations of the eigenvalues near the edge.

\subsection{Outline of the proof}
In this section, we outline the proof of Theorem \ref{thm: clt}. The rest of the paper contains the derivation of the auxiliary results used here.

By Lemma \ref{lem: representation}, we have the representation
\begin{equation}\label{eqn: double-contour}
  \langle e^{tR_{12}} \rangle = \frac{\int_{\gamma-i\infty}^{\gamma+i\infty} \int_{\gamma-i\infty}^{\gamma+i\infty}e^{\frac{N}{2}\tilde{G}(z,w)}\mathrm{d}z \mathrm{d}w}{\Big (\int_{\gamma-i\infty}^{\gamma+i\infty} e^{\frac{N}{2}G(z)} dz\Big )^2},
  \end{equation}
  where $G(z,w)$ and $\tilde{G}(z,w)=\tilde{G}(z,w,t)$ are defined in \eqref{eqn: Gtilde} and \eqref{eqn: Gdef}, respectively. The contour integral appearing in the denominator is the partition function of the model. It was computed by Baik and Lee \cite{baiklee} for $\beta<1$ fixed, via a steepest descent approximation. The essential point in \cite{baiklee} is that the saddle point equation
  \begin{equation}\label{eqn: saddlepoint}
    G'(z)=0
  \end{equation}
  can be solved to high accuracy, because it involves random matrix quantities. In particular, when $\beta<1$, the solution $\gamma$ of \eqref{eqn: saddlepoint} with greatest real part is well approximated by the solution $\hat{\gamma}$ of
  \[m_\mathrm{sc}(z)=2\beta.\]
This remains true for $\beta_N=1-cN^{-1/3+\tau}$. In this case, $\gamma$ is now found with high probability at distance $N^{-2/3+2\tau^-}$ from the edge of the semi-circle and the singularities of $G(z)$. See Proposition \ref{lem: gamma-dist}. In particular, the derivatives of $G$ which appear in the saddle point approximation are now large, which complicates the analysis.

Having determined the approximate location of the saddle point and its distance to the singularities, we can apply the steepest descent approximation. The denominator in \eqref{eqn: double-contour} is computed in Proposition \ref{prop: taylor}:
\[\int_{\gamma-i\infty}^{\gamma+i\infty} e^{\frac{N}{2}G(z)} dz=i\sqrt{\frac{4\pi}{N G''(\gamma)}}(1+O(N^{-c})).\]
To compute the numerator, we we expand the function $\tilde{G}(w,z)$ as follows:
\[\tilde{G}(w,z)=G(z)+G(w)+\frac{t^2(1-\beta^2)}{4\beta^2N^2}\frac{m_N(w)-m_N(z)}{w-z}+\text{(lower order)}.\]
See Proposition \ref{prop: Gt-exp}. We then compute the double contour integral in the numerator of \eqref{eqn: double-contour} by successive applications of the steepest descent approximation around $z=w=\gamma$. 

It then remains only to compute the quantity $m_{\mathrm{sc}}(\gamma)$ to identify the limiting variance. This is done in Lemma \ref{lem: variance}.

\section{Notation and auxiliary results from random matrix theory}
Throughout, we denote by $\mathbf{P}$ the joint distribution of the disorder variables $g_{ij}$, $1\le i,j\le N$. We say that a sequence of events $A_N$ occurs with \emph{overwhelming probability} if, for any $D>0$, $\mathbf{P}(A_N^c)\le N^{-D}$ for all sufficiently large $N$.

The eigenvalues of the matrix $(M_{ij})=(1/2)(g_{ij}+g_{ji})$ are real and almost surely distinct. We label them in decreasing order as $\lambda_1\ge \lambda_2\ge ...\ge \lambda_N$. The key input from random matrix theory we will require is the \emph{local semi-circle law}, describing the behavior of the resolvent matrix $(M-z)^{-1}$. It is a classical fact, originally due to E. Wigner, that the empirical distribution of the eigenvalues of $M$ converges weakly to the semi-circle distribution. With our scaling, that distribution has density
\[\rho_{\mathrm{sc}}(x)=\frac{2}{\pi}\sqrt{1-x^2}.\]
The local semi-circle law considerably strengthens this statement, giving a sharp estimate for the distance between the trace of the resolvent matrix
\begin{equation}\label{eqn: mN-def}
m_N(z)=\frac{1}{N}\sum_{j=1}^N \frac{1}{\lambda_j-z}=\frac{1}{N}\mathrm{tr}(M-z)^{-1}
\end{equation}
and the corresponding quantity for the semi-circle distribution:
\[m_{\mathrm{sc}}(z)=\int \frac{1}{x-z}\rho_{\mathrm{sc}}(x)\,\mathrm{d}x.\]
The statement below is adapted from \cite{knowlesBG}:
\begin{theorem} \label{thm: sc-law} Fix $\delta>0$ and define the domain
\[S_N(\delta)=\{E+i\eta: |E|\le \delta^{-1}, N^{-1+\delta}\le \eta\le \delta^{-1}\}.\]
Then for any $\epsilon>0$, with overwhelming probability, we have
\[|m_N(z)-m_{\mathrm{sc}}(z)|\le N^{\epsilon}(N\eta)^{-1}\]
uniformly in $z$.
\end{theorem}

A consequence of the local semi-circle law is the following \emph{eigenvalue rigidity} statement.
\begin{proposition}
For $1\le i \le N$, define $\gamma_i$, the \emph{typical location} of $\lambda_i$, by
\[N\int_{\gamma_i}^1\rho_{\mathrm{sc}}(x)\,\mathrm{d}x=i-\frac{1}{2}.\]

With overwhelming probability, for any $\epsilon>0$, we have
\begin{equation}\label{eqn: rigidity}
|\lambda_i-\gamma_i|\le N^{-2/3+\epsilon}(i\wedge (N+1-i))^{-1/3},
\end{equation}
for $i=1,\ldots, N$.

In particular, we have the following bound for the empirical density, the number of eigenvalues in any interval of size $\gg N^{-1}$: for any $\epsilon>0$, with overwhelming probability, we have
\begin{equation}\label{eqn: density}
\frac{1}{N}\#\{i:\lambda_i\in I\} = \int_I \rho_{\mathrm{sc}}(x)\,\mathrm{d}x+O(N^{-1+\epsilon})
\end{equation}
uniformly for all intervals $I\subset \mathbb{R}$.
\end{proposition}

\section{A representation formula}
In this section, we give an explicit contour integral representation for the the overlap $R_{12}$. 
\subsection{Laplace transform of the overlap}
Define the matrix 
\[M_{ij}=\frac{g_{ij}+g_{ji}}{2}, \quad 1\le i,j\le N.\]
The matrix $M$ is a symmetric matrix with mean-zero Gaussian, with variance 1/2 off the diagonal and variance 1 on the diagonal. The entries are independent for the symmetry constraint $M_{ij}=M_{ji}$. $M$ is thus a rescaling (by $1/\sqrt{2}$) of the classical GOE (Gaussian Orthogonal Ensemble) matrix. In particular, $M$ is diagonalizable, with real eigenvalues. 

In terms of $M$, the Laplace transform of the overlap distribution, introduced in \eqref{computation:Laplace_R} is expressed as follows:
\begin{align}\label{for:laplace}
\langle e^{tR_{12}} \rangle & = \frac{1}{Z_N^2}\frac{1}{ |\mathbb{S}^{N-1}|^2} \int_{(\mathbb{S}^{N-1})^2} \exp\bigg(\beta \langle \mathbf{x}, M\mathbf{x}\rangle+\beta \langle \textbf{y}, M\mathbf{y}\rangle +\frac{t}{N}\langle \mathbf{x},\textbf{y} \rangle \bigg) \mathrm{d}\omega(\textbf{x}) \mathrm{d} \omega(\mathbf{y}).
\end{align}

\begin{lemma}\label{lem: representation} Let $\beta>0$. The Laplace transform of the overlap, $R_{12}$, has the following contour integral representation:
\begin{equation}\label{eqn: josephus}
\langle e^{tR_{12}} \rangle = \frac{\int_{\gamma-i\infty}^{\gamma+i\infty} \int_{\gamma-i\infty}^{\gamma+i\infty}e^{\frac{N}{2}\tilde{G}(z,w)}\mathrm{d}z \mathrm{d}w}{\Big (\int_{\gamma-i\infty}^{\gamma+i\infty} e^{\frac{N}{2}G(z)} dz\Big )^2},
\end{equation}
where 
\begin{equation}\label{eqn: Gdef}
G(z)= 2\beta z - \frac{1}{N}\sum_i \log (z-\lambda_i),
\end{equation}
and 
\begin{equation}\label{eqn: Gtilde}
\tilde{G}(z,w,t)= 2\beta (z+w) - \frac{1}{N}\sum_i \log \Big( (z-\lambda_i)(w-\lambda_i)- \frac{t^2}{4\beta^2 N^2}\Big ).
\end{equation}
\end{lemma}

\begin{proof}
From \cite[Eqn. (4.7)]{baiklee}, we know that:
\begin{equation}
Z_N(\beta)= \frac{\Gamma(N/2)}{2\pi i(N\beta)^{N/2-1}} \int_{\gamma-i\infty}^{\gamma+i\infty} e^{\frac{N}{2}G(z)} dz,
\end{equation}
where $\gamma>\lambda_1$, and 
\begin{equation}
G(z)= 2\beta z - \frac{1}{N}\sum_{i=1}^N \log (z-\lambda_i).
\end{equation}
 Now we rewrite the integral (\ref{for:laplace}). Let $S^{N-1} = \{ \mathbf{x} \in \mathbb{R}^N: |\mathbf{x}| =1\}$, the unit sphere in $\mathbb{R}^N$. Let $\mathrm{d}\Omega$ be the surface area measure on $S^{N-1}$, then $\frac{\mathrm{d}\Omega}{\vert S^{N-1}\vert}$ is the uniform measure on $S^{N-1}$. After rescaling, the numerator in \eqref{for:laplace} is
 \[\frac{1}{\vert S^{N-1}\vert^2}\int_{( S^{N-1})^2} \exp\bigg(\beta N \langle \mathbf{x}, M\mathbf{x}\rangle+\beta N\langle \mathbf{y}, M\mathbf{y}\rangle +{t}\langle \textbf{x},\textbf{y} \rangle \bigg) \mathrm{d}\Omega(\mathbf{x}) \mathrm{d} \Omega(\mathbf{y}).\]
The matrix $M$ is diagonalized by an orthogonal transformation leaving $\mathrm{d}\Omega$ invariant, so the last quantity equals
\[\frac{1}{\vert S^{N-1}\vert^2}\int_{( S^{N-1})^2} \exp\bigg(\beta N \sum_{i=1}^N \lambda_i (x_i^2+y_i^2) +{t}\langle \textbf{x},\textbf{y} \rangle \bigg) \mathrm{d}\Omega(\mathbf{x}) \mathrm{d} \Omega(\mathbf{y}).\]

In order to compute this integral, we consider 
\begin{align}\label{def:J}
J(z,w)=\int_{\mathbb{R}^N}\int_{\mathbb{R}^N} e^{\beta N\sum_{i=1}^N (\lambda_i-z) x_i^2} e^{\beta N\sum_{i=1}^N (\lambda_i -w) y_i^2} e^{t\sum_{i=1}^N x_i y_i}\,\prod_{i=1}^N \mathrm{d}x_i\mathrm{d}y_i.
\end{align}
This integral is absolutely convergent if 
\[\Re z, \Re w > \lambda_1.\]
In this region, $J(z,w)$ defines an analytic function of $z$ and $w$.
We pass to polar coordinates: substitute $\mathbf{x}=r_1\mathbf{x}_1$ and $\mathbf{y}=s_1\mathbf{y}_1$ with $r_1,s_1>0$ and $|x_1|=|y_1|=1$, and then set $\beta Nr_1^2=r$, $\beta Ns_1^2=s$
to find that 
\begin{align*}
 J(z,w)=\frac{1}{4(\beta N)^N}\int_0^\infty \int_0^\infty e^{- z r}e^{- w s} I(r,s)s^{\frac{N}{2}-1}r^{\frac{N}{2}-1}\,\mathrm{d}r\mathrm{d}s,
\end{align*}
where
\[I(r,s)=\int_{S^{N-1}\times S^{N-1}} e^{ r\langle \mathbf{x}_1, M\mathbf{x}_1\rangle +s\langle \mathbf{y}_1, M\mathbf{y}_1\rangle}   e^{\frac{t\sqrt{rs}}{\beta N} \langle\mathbf{x}_1,\mathbf{y}_1\rangle }\,\mathrm{d}\Omega_1 \mathrm{d}\Omega_2.\]
Completing the square, we obtain
\begin{align*}
&\beta N(\lambda_i-z) x_i^2 + t x_iy_i\\ 
=&\beta N(\lambda_i-z) \left( x_i^2 + 2 x_i  \cdot \frac{t}{2\beta N} \frac{y_i}{\lambda_i - z}+\frac{t^2}{4\beta^2 N^2}\frac{y_i^2}{(\lambda_i-z)^2}\right)-\frac{t^2}{4\beta N}\frac{y_i^2}{\lambda_i-z}\\
=&\beta N(\lambda_i-z)\left(x_i+\frac{t}{2\beta N(\lambda_i-z)}y_i\right)^2-\frac{t^2}{4\beta N } \frac{y_i^2}{\lambda_i-z}.
\end{align*}
Shifting the $x_i$, the integral (\ref{def:J}) is given by
\begin{align}
&\int e^{\beta N\sum_{i=1}^N (\lambda_i-z) x_i^2} e^{\beta N\sum_{i=1}^N \left( \lambda_i -w -\frac{t^2}{4\beta^2 N^2}\frac{1}{\lambda_i-z}\right) y_i^2}\,\prod_{i=1}^N \mathrm{d}x_i\mathrm{d}y_i\\
=& \Big ( \frac{\pi}{\beta N}\Big )^N \left(\prod_{i=1}^N (z-\lambda_i)(w-\lambda_i+\frac{t^2}{4\beta^2 N^2}\frac{1}{\lambda_i-z})\right)^{-1/2}\\
=& \Big ( \frac{\pi}{\beta^2 N}\Big )^N  \big(\prod_{i=1}^N (w-\lambda_i)(z-\lambda_i)-\frac{t^2}{4\beta^2 N^2}\big)^{-1/2}.\label{for:laplace_transform}
\end{align}
$(\cdot )^{-1/2}$ is defined in terms of the principal branch of the logarithm.

Taking the inverse Laplace transform and using (\ref{for:laplace_transform}), we obtain that
\begin{align}
&\frac{I(r,s)s^{\frac{N}{2}-1}r^{\frac{N}{2}-1}}{4(\beta N)^N}\\
=&\frac{1}{(2\pi i)^2}\int_{\gamma-i\infty}^{\gamma+i\infty} \int_{\gamma-i\infty}^{\gamma+i\infty} e^{zr}e^{ws} J(z,w) \mathrm{d}z \mathrm{d}w\\
=&\Big ( \frac{\pi}{\beta N}\Big )^N \frac{1}{(2\pi i)^2}\int_{\gamma-i\infty}^{\gamma+i\infty} \int_{\gamma-i\infty}^{\gamma+i\infty} e^{zr}e^{ws} \big(\prod_{i=1}^N (w-\lambda_i)(z-\lambda_i)-\frac{t^2}{4\beta^2 N^2}\big)^{-1/2} \mathrm{d}z \mathrm{d}w,
\end{align}
where $\gamma$ is a real number satisfying $\gamma >\lambda_1+ \frac{t}{2\beta N}$. By letting $r=s=\beta N$, we obtain:
\begin{align*}
&\int_{(S^{N-1})^2} e^{ N\beta\langle \mathbf{x}_1, M\mathbf{x}_1\rangle +N\beta\langle \mathbf{y}_1, M\mathbf{y}_1\rangle}   e^{t \langle\textbf{x}_1,\textbf{y}_1\rangle }\,\mathrm{d}\Omega_1 \mathrm{d}\Omega_2 \\=&\frac{4\pi^N}{(\beta N)^{N-2}}\int_{\gamma-i\infty}^{\gamma+i\infty} \int_{\gamma-i\infty}^{\gamma+i\infty} e^{N\beta(z+w)}\big(\prod_{i=1}^N (w-\lambda_i)(z-\lambda_i)-\frac{t^2}{4\beta^2 N^2}\big)^{-1/2} \mathrm{d}z \mathrm{d}w
\end{align*}
\end{proof}

\section{Estimates for $G(z)$ and derivatives}\label{sec: gamma}
To perform a saddle point approximation of the integrals appearing in \eqref{eqn: josephus}, we estimate the quantity $G(z)$ \eqref{eqn: Gdef}
for $z$ close to the largest solution $\gamma=\gamma(\beta)\in \mathbb{R}$ of 
\begin{equation}\label{eqn: Gprime}
G'(z)=2\beta -\frac{1}{N}\sum_{i=1}^N \frac{1}{z-\lambda_i}=2\beta-m_N(z).
\end{equation}
We refer to $\gamma$ as as ``the'' saddle point. As $\beta$ increases from a fixed value less than the critical value $\beta_c=1$ to $\beta_c-N^{-1/3}$, the distance between $\gamma$ and the edge of the semi-circle distribution (1 in our scaling) decreases from order $1$ to order $N^{-2/3}$. 

We use a version of the semi-circle law outside the limiting spectrum. First, we introduce some notation:
\begin{align}
m_{\mathrm{sc}}(z)&= 2(-z+\sqrt{z^2-1}),\\
\kappa &= ||E|-1|. \nonumber
\end{align}
\begin{theorem}\label{thm: outside-sc}
Define the domain
\[S=\{E+i\eta: |E|\ge 1+N^{-2/3+\epsilon},\eta>0\}.\]
Then for any $\delta>0$,
\begin{equation}\label{eqn: sc_law}
\left|m_N(z)-m_{\mathrm{sc}}(z)\right|\le \frac{1}{N^{1-\delta}}\frac{1}{(\kappa+\eta)+(\kappa+\eta)^2}, \quad z\in S
\end{equation}
with overwhelming probability.
\end{theorem}

Our initial step is to obtain an estimate for this distance, depending of $\beta_N$, which we assume is of the form
\[\beta_N=1-N^{-2/3+\tau}.\]

\begin{lemma}\label{lem: gamma-dist}
Let $\beta_N=1-cN^{-1/3+\tau}$, where $0<\tau<1/3$. With overwhelming probability, for any $\epsilon>0$, the unique solution $\gamma$ in $(\lambda_1,\infty)$ to
\[G'(z)=0\]
satisfies
\begin{equation}\label{eqn: gamma-dist}
  N^{-2/3+2\tau-\epsilon}<\gamma-\lambda_1 < N^{-2/3+2\tau+\epsilon}
\end{equation}
for $N\ge N_0(\epsilon)$.

\begin{proof}
For real $z$, the expression \eqref{eqn: Gprime} represents an increasing function on $(\lambda_1,\infty)$. As noted in \cite[Lemma 6.1]{baiklee}, we always have $G'(\lambda_1+\frac{1}{3\beta N})<0$, so
\[\gamma-\lambda_1> \frac{1}{3\beta N}.\]
Let $\eta>0$ be such that $\tau <\eta< 2\tau$, and set
\[x_N=\lambda_1+cN^{-2/3+\eta}.\]

By \eqref{eqn: sc_law}, for any $\delta>0$ and $N$ large enough, we have
\begin{equation}\label{eqn: Gprime-2}
\begin{split}
G'(\lambda_1+cN^{-2/3+\eta})= 2(1-cN^{-1/3+\tau})-2(x_N-\sqrt{x_N^2-1})+O(N^{-1/3+\delta}).
\end{split}
\end{equation}
We choose $\delta<\eta/2$.

By eigenvalue rigidity, we have
\begin{equation}\label{eqn: lambda1}
|1-\lambda_1|\le CN^{-2/3+\delta}
\end{equation}
with high probability, so
\begin{align*}
x_n-\sqrt{x_N^2-1}&=\lambda_1+cN^{-2/3+\eta}-\sqrt{(cN^{-2/3+\eta}+\lambda_1-1)(1+\lambda_1+cN^{-2/3+\eta})}\\
&=\lambda_1+cN^{-2/3+\eta}-(\sqrt{2}+O(N^{-2/3+\eta}))c^{1/2}N^{-1/3+\eta/2}\\
&= 1- \sqrt{2c}N^{-1/3+\eta/2}+cN^{-2/3+\eta} +O(N^{-2/3+\delta}).
\end{align*}
with $c'>0$. Putting this into \eqref{eqn: Gprime-2}, we obtain:
\[G'(\lambda_1+cN^{-2/3+\eta})< -2cN^{-1/3+\tau}+o(N^{-1/3+\tau})<0,\]
for $N$ sufficiently large. By \eqref{eqn: lambda1}, this implies 
\[\gamma-\lambda_1>cN^{-2/3+\eta}>cN^{-2/3+2\tau-\epsilon}\]
for any $\epsilon$ sufficiently small.
By the same argument one shows
\[\gamma-\lambda_1 < N^{-2/3+2\tau+\epsilon}. \]
\end{proof}
\end{lemma}

\begin{proposition}[Estimates on the derivatives]
  Let $\beta_N=1-N^{-1/3+\tau}$ for $\tau>0$, and let $\gamma$ satisfy $\gamma-\lambda_1 \ge N^{-2/3+2\tau-\epsilon}$. Let $s\in \mathbb{R}$. There is a constant such that with overwhelming probability,
\begin{equation}\label{eqn: bounded}
|G'(\gamma+is)|\le CN^{\epsilon'}.
\end{equation}
More generally,with overwhelming probability, for any integer $k\ge 2$:
\begin{equation}\label{eqn: Gk_bound}
|G^{(k)}(\gamma+is)|\le CN^{2k/3 -(2k-3)\tau-1+\epsilon'}.
\end{equation}
\begin{proof}
The bound \eqref{eqn: bounded} follows by the local semicircle law \eqref{eqn: sc_law}. The estimate \eqref{eqn: Gk_bound} is derived by a standard argument using the using \eqref{eqn: density} and the lower bound on $\gamma-\lambda_1$.
\end{proof}
\end{proposition}

\section{Saddle point approximation}
\begin{proposition}\label{prop: Gt-exp}
Suppose $\lambda_1-\gamma>N^{-2/3+2\tau-\epsilon}$, and
\begin{equation}\label{eqn: parameters}
\begin{split}
\beta&=\beta_N=1-cN^{-1/3+\tau},\\
z&=\gamma+ir,\\
w&=\gamma+is,
\end{split}
\end{equation}
where $r,s\in \mathbb{R}$.

There is a $c(\tau)>0$ such that with overwhelming probability:
\[e^{\frac{N}{2}\tilde{G}(z,w)}=e^{\frac{N}{2}(G(z)+G(w))}\exp\left(\frac{t^2(1-\beta^2_N)}{4\beta^2}\frac{m_N(w)-m_N(z)}{w-z}\right)(1+O(N^{-c(\tau)})). \]
\begin{proof}
Replacing $t$ by $\sqrt{N(1-\beta_N^2)}t$ in \eqref{eqn: Gtilde}, the logarithm in the definition of $\tilde{G}(z,w)$ is
\begin{equation*}
  \log\big((z-\lambda_i)(w-\lambda_i)-\frac{t^2(1-\beta^2)}{4\beta^2 N}\big)= \log((z-\lambda_i)(w-\lambda_i))+\log\big(1-\frac{t^2(1-\beta^2)}{4\beta^2 N(z-\lambda_i)(w-\lambda_i)}\big).
\end{equation*}
By Taylor expansion, the second term is
\begin{equation}\log((z-\lambda_i)(w-\lambda_i))-\frac{t^2(1-\beta^2)}{4\beta^2 N(z-\lambda_i)(w-\lambda_i)}+\epsilon_{N,i}.
\end{equation}
Here,
\[|\epsilon_{N,i}|\le \frac{C}{N^{8/3-2\tau}}\frac{1}{|\lambda_i-z|^2|\lambda_i-w|^2},\]
since $1-\beta_N^2=O(N^{-1/3+\tau})$.

By \eqref{eqn: density} and the lower bound for $\lambda_1-\gamma$ \eqref{eqn: gamma-dist},
\begin{align}
\sum_{i=1}^N |\epsilon_{N,i}|&\le N^{-5/3+2\tau+\epsilon}\sum_{k=1}^{\big(\frac{2}{3}-2\tau\big)\log N}2^{5/2k}\\
&\le CN^{-2\tau+3\epsilon}.
\end{align}

Summing over $i$, this gives
\begin{align*}
&\sum_{i=1}^N \log((z-\lambda_i)(w-\lambda_i))-\frac{t^2(1-\beta^2)}{4\beta^2 N}\sum_{i=1}^N \frac{1}{(z-\lambda_i)(w-\lambda_i)}+\epsilon'_N\\
=&\sum_{i=1}^N \log((z-\lambda_i)(w-\lambda_i))+\frac{t^2(1-\beta^2)}{4\beta^2}\frac{m_N(w)-m_N(z)}{w-z}+\epsilon'_N,
\end{align*}
where $\epsilon_N'\le CN^{-2\tau+3\epsilon}$.

It follows that, for $z,w\in \gamma+i\mathbb{R}$,
\[e^{\frac{N}{2}\tilde{G}(z,w)}=e^{\frac{N}{2}(G(z)+G(w))}e^{-\frac{t^2(1-\beta^2_N)}{4\beta^2}\frac{m_N(w)-m_N(z)}{w-z}}\cdot (1+O(N^{-c})),\]
as claimed.
\end{proof}
\end{proposition}

\begin{lemma}
Assume the restrictions \eqref{eqn: parameters} on the parameters hold. In addition that
\begin{equation}
|s|, |r|\le N^{-2/3+\delta_1}.
\end{equation}
Suppose also that
\[\tau<\delta_1<2\tau\]
is sufficently close to $2\tau$. Then with overwhelming probability,
\begin{equation}\label{eqn: quotient-bound}
\left|\frac{m_N(z)-m_N(w)}{z-w}-m_{\mathrm{sc}}(\gamma)\right|\le N^{-c(\tau)}
\end{equation}
for any $c(\tau)<2\tau-\delta_1$ and $N\ge N_0(c)$.
\begin{proof}
We have
\begin{equation}\label{eqn: mzw-quotient}
  \begin{split}
  \left|\frac{m_N(z)-m_N(w)}{z-w}-m_N'(\gamma)\right|&\le \left|\int_0^1\big( m'_N(z+s(z-w))-m_N'(\gamma)\big)\,\mathrm{d}s\right|\\
  &\le \max_{\zeta\in D(\tau,\delta_1) }|m''_N(\zeta)|\cdot N^{-2/3+\delta_1},
  \end{split}
\end{equation}
where
\[D(\tau,\delta_1):=\{\zeta=\gamma+iu: |u|\le N^{-2/3+\delta_1}\}.\]
By \eqref{eqn: bounded}, $\max_{D(\tau,\delta_1)}|m''_N(\zeta)|\le N^{1-3\tau+\epsilon}$ with overwhelming probability, where $\epsilon>0$ is arbitrary and $N$ large. So
\begin{align*}
(1-\beta_N^2)\left|\frac{m_N(z)-m_N(w)}{z-w}-m_N'(\gamma)\right|&\le CN^{-1/3+\tau} N^{-2/3+\delta_1}\cdot N^{1-3\tau+\epsilon} \\
&\le C N^{-2\tau+\delta_1+\epsilon}.
\end{align*}

By \eqref{eqn: sc_law} with $\kappa=N^{-2/3+2\tau-\epsilon}$, we have for $r=N^{2/3-\tau+\epsilon}$:
\begin{equation}\label{eqn: deriv-diff}
\begin{split}
  (1-\beta^2_N)|m_N'(\gamma)-m'_{\mathrm{sc}}(\gamma)|&=(1-\beta_N^2)\left|\frac{1}{\pi i}\oint_{|\gamma-w|=r}\frac{m_N(w)-m_{\mathrm{sc}}(w)}{(\gamma-w)^2}\,\mathrm{d}w\right|\\
                                      &\le CN^{-1/3+\tau}N^{-1+\delta}N^{2/3-2\tau+\epsilon} N^{2/3-\tau+\epsilon}\\
  &\le CN^{-2\tau+\delta+2\epsilon}.
\end{split}
\end{equation}
Here, we have used $\gamma-\lambda_1\ge N^{-2/3+\tau-\epsilon}$.
\end{proof}
\end{lemma}

\subsection{Control of the region outside the saddle point}
In this section, we control the quantity
\[\int_{\gamma-i\infty}^{\gamma+i\infty} e^{N(G(\gamma+it)-G(\gamma))}\,\mathrm{d}t\]
in the region away from the saddle point under the assumption of eigenvalue rigidity. In particular, we have, for any $\epsilon>0$ and large $N$:
\[\lambda_1-\gamma \ge N^{-2/3+2\tau-\epsilon}.\]
\begin{proposition}[Away from the saddle point]
We have
\begin{equation}\label{eqn: int-away}
  \int_{\{s:|s|>N^{-2/3+\delta_1}\}} e^{\frac{N}{2}(G(\gamma+i s)-G(\gamma))}\,\mathrm{d}s\le C\exp(-N^{c(\tau)}),
  \end{equation}
for some $c(\tau)>0$.
\begin{proof}
 As in \cite[Lemma 6.3]{baiklee}, the integrand in \eqref{eqn: int-away} is bounded by
\begin{equation}
\label{eqn: baiklee-outside}
\exp\left(-\frac{1}{4}\sum_{j=1}^N\log\left(1+\frac{s^2}{(\gamma-\lambda_j)^2}\right)\right).
\end{equation}

Since $\gamma-\lambda_1\le N^{-2/3+\tau+\epsilon}$, by \eqref{eqn: density}, for any $\delta>0$ we have with overhwelming probability:
\begin{equation}\label{eqn: ev-counts}
\#\{j: |\lambda_j-\gamma|\le CN^{-2/3+2\tau}\} \ge N^{3\tau-2\epsilon},
\end{equation}
for any $\epsilon>0$ and large enough $N$.
This implies
\[\sum_{j=1}^N \log\left(1+\frac{s^2}{(\gamma-\lambda_j)^2}\right)\ge N^{3\tau-2\epsilon}\log(1+N^{2(\delta_1-\tau)}),\]
whence
\[\int_{\{s:N^{-2/3+\delta_1}<|s|\le 10\}} e^{\frac{N}{2}(G(\gamma+i s)-G(\gamma))}\,\mathrm{d}s\le C\exp(-N^{3\tau-2\epsilon})\]
for large enough $N$.

Since $|\lambda_j-\gamma|\le 4$ with overwhelming probability \cite[Theorem H.1]{knowlesBG}, the integral over the region $|s|\ge 10$ is also exponential small.
\end{proof}
\end{proposition}

\subsection{Approximation near the saddle point}

\begin{proposition}[Approximation in the central region]\label{prop: taylor}
With overwhelming probability,
\[\int_{\{|\Im z|\le N^{-2/3+\delta_1}\}} e^{\frac{N}{2}(G(z)-G(\gamma))}\,\mathrm{d}z=i\sqrt{\frac{4\pi}{G''(\gamma)}}(1+O(N^{-c}))\]
for some $c=c(\tau)>0$.
\begin{proof}
Since $s\in D(\tau,\delta_1)$, using \eqref{eqn: Gk_bound}, we have by Taylor expansion:
\begin{equation}\label{eqn: G''}
  \begin{split}
N(G(\gamma+is)-G(\gamma))&= -\frac{s^2}{2}G''(\gamma)+\frac{N^{-1/2}|s|^3}{3}\max_{s\in D(\tau,\delta_1)}|G^{(3)}(\gamma+is)|\\
                         &=-\frac{s^2}{2}G''(\gamma)+|s|^3\cdot O(N^{1/2-6\tau+3\epsilon})\\
                         &=-\frac{s^2}{2}G''(\gamma)+O(N^{-3(2\tau-\delta_1-\epsilon)}).
                         \end{split}
\end{equation}
By \eqref{eqn: density}, we have the following lower bound for $G''(\gamma)$:
\begin{equation}\label{eqn: G''-lwr-bound}
\begin{split}
  G''(\gamma)&\ge \frac{1}{N}\sum_{i=1}^N \frac{1}{(\gamma-\lambda_i)^2} \\
             &\ge \frac{1}{N}\sum_{\{i:|\lambda_i-\gamma|\le N^{-2/3+2\tau}\}}\frac{1}{(\gamma-\lambda_i)^2}\\
             &\ge N^{4/3-4\tau-1}\#\{i:|\lambda_i-\gamma|\le N^{-2/3+2\tau}\}\\
  &\ge N^{1/3-\tau-\epsilon}.
\end{split}
\end{equation}

Rescaling $s=\Im z$ in the contour integral, we obtain
\[iN^{-1/2}\int_{|s|\le N^{-1/6+\delta_1}} e^{\frac{N}{2}(G(\gamma+i\frac{s}{N^{1/2}})-G(\gamma))}\,\mathrm{d}s.\]
Inserting the expansion \eqref{eqn: G''} into the integral, we obtain the two terms
\[\left(\int_{|s|\le N^{-1/6+\tau/2+\epsilon/4} }+\int_{N^{-1/6+\tau/2+\epsilon/4}<|s|\le N^{-1/6+\delta_1}}\right)  e^{\frac{N}{2}(G(\gamma+i\frac{s}{N^{1/2}})-G(\gamma))} \,\mathrm{d}s =\mathrm{I}+\mathrm{II},\]
with
\[ \mathrm{I}=\sqrt{\frac{4\pi}{G''(\gamma)}}\big(1+O(N^{-c'})\big),\quad  \mathrm{II}=e^{-cN^{c''}}.\]
\end{proof}
\end{proposition}

\begin{proposition}[Saddle point approximation of the numerator]\label{prop: saddle}
  With overwhelming probability,
  \[\int_{\gamma-i\infty}^{\gamma+i\infty} \int_{\gamma-i\infty}^{\gamma+i\infty}e^{\frac{N}{2}(\tilde{G}(z,w)-2G(\gamma))}\,\mathrm{d}z\mathrm{d}w=-\frac{4\pi}{NG''(\gamma)}e^{\frac{t^2(1-\beta_N^2)}{2\beta_N^2}m'_{\mathrm{sc}}(\gamma)} \left(1+O(N^{-c(\tau)})\right),\]
  for some $c(\tau)>0$.
\begin{proof}
 Note the estimate
 \begin{equation}\label{eqn: mq-bound}
\left|\frac{m_N(w)-m_N(z)}{w-z}\right|\le CN^{\epsilon}
\end{equation}
 for $\epsilon>0$ arbitrary and $N$ large. This follows from \eqref{eqn: bounded} using \eqref{eqn: mzw-quotient}.

\begin{equation}\label{eqn: 4-terms}
\begin{split}
  \int_{\gamma-i\infty}^{\gamma+i\infty} \int_{\gamma-i\infty}^{\gamma+i\infty}e^{\frac{N}{2}(\tilde{G}(z,w)-2G(\gamma))}\,\mathrm{d}z\mathrm{d}w =& \int_{|s|\le N^{-2/3+\delta_1}}\int_{|r|\le N^{-2/3+\delta_1}} \cdots \\
                                                                                                                                                  +& \int_{|s|\le N^{-2/3+\delta_1}}\int_{|r|> N^{-2/3+\delta_1}}\cdots \\
                                                                                                                                                 +&\int_{|s|> N^{-2/3+\delta_1}}\int_{|r|\le N^{-2/3+\delta_1}}\cdots\\
  +& \int_{|s|> N^{-2/3+\delta_1}}\int_{|r|> N^{-2/3+\delta_1}} \cdots.
\end{split}
\end{equation}

The first three terms in \eqref{eqn: 4-terms} are bounded by $C\exp(-N^{c(\tau)})$ by \eqref{eqn: mq-bound}, \eqref{eqn: int-away} and Proposition \ref{prop: taylor}.

For the first term, we have by \eqref{eqn: quotient-bound},
\begin{align}
  &\int_{|s|\le N^{-2/3+\delta_1}}\int_{|r|\le N^{-2/3+\delta_1}}e^{\frac{N}{2}\tilde{G}(z,w)}\,\mathrm{d}w\mathrm{d}z \nonumber \\
  = &e^{\frac{t^2(1-\beta_N^2)}{4\beta_N^2}m'_{\mathrm{sc}}(\gamma)} \int_{|s|\le N^{-2/3+\delta_1}}\int_{|r|\le N^{-2/3+\delta_1}} e^{\frac{N}{2}(G(z)+G(w)-2G(\gamma))}\,\mathrm{d}w\mathrm{d}z \label{eqn: central} \\
  +&O(N^{-2\tau+\delta_1+\epsilon})\int_{|s|\le N^{-2/3+\delta_1}}\int_{|r|\le N^{-2/3+\delta_1}} |e^{\frac{N}{2}(G(z)+G(w)-2G(\gamma))}|\,\mathrm{d}w\mathrm{d}z. \label{eqn: off-center}
\end{align}
For \eqref{eqn: central}, we use Proposition \ref{prop: taylor}. For \eqref{eqn: off-center}, we use the Taylor expansion \eqref{eqn: G''} to replace the region of integration by $\{|s|, |r|\le N^{-2/3+\tau+\epsilon}\}$ (with an exponentially small error), and then find the bound
\[\eqref{eqn: off-center}=O(N^{-4/3+2\epsilon}).\]
Finally, by \eqref{eqn: G''-lwr-bound}, we have
\[\frac{1}{NG''(\gamma)}=O(N^{-4/3+\tau+\epsilon}),\]
and the result follows from this.
\end{proof}
\end{proposition}

\section{Proof of Theorem \ref{thm: clt}}
We start by computing the asymptotic variance of the overlap.
\begin{lemma}[Computation of the variance]\label{lem: variance}
Let $\gamma$ be the saddle point defined in Section \ref{sec: gamma}. With overwhelming probability, we have
\begin{equation}
e^{\frac{t^2}{4\beta^2}m'_{\mathrm{sc}}(\gamma)}=e^{t^2}\big(1+O(N^{-c(\tau)})\big)
\end{equation}
for some $c(\tau)>0$.
\begin{proof}
We begin by approximating the saddle point $\gamma$ by the quantity
\[\hat{\gamma}=\frac{1}{2}\left(\beta_N+\frac{1}{\beta_N}\right),\]
the solution of 
\[m_{\mathrm{sc}}(x)=2(-x+\sqrt{x^2-1})=2\beta_N.\]
Note that for $\beta_N=1-cN^{-1/3+\tau}$, we have
\[\hat{\gamma}= 1+c^2N^{-2/3+2\tau}+O(N^{-2+3\tau}).\]

Next, we note:
\[ m_{sc}'(x)=-2+\frac{2x}{\sqrt{x^2-1}}.\]
By \eqref{eqn: sc_law}, combined with the derivative bound 
\[m'_{\mathrm{sc}}(x)> N^{1/3-\tau+\epsilon/2}\]
in the region $x>1+N^{-2/3+2\tau-\epsilon}$, a Taylor expansion of $G'(z)$ about $\hat{\gamma}$ (see for example \cite[Corollary 5.2]{baiklee}) shows that $G'(\gamma)=0$ for some $\gamma>\lambda_1$ with
\begin{equation}
|\hat{\gamma}-\gamma|\le CN^{-2/3-2\tau+\epsilon},
\end{equation}
with overwhelming probability, for any $\epsilon>0$.
From this, we find
\[(1-\beta_N^2)|m_{\mathrm{sc}}'(\gamma)-m_{\mathrm{sc}}'(\hat{\gamma})|\le N^{-c(\tau)}.\]
It follows that we can replace $\gamma$ with $\hat{\gamma}$ at the price of a multiplicative error of order $(1+O(N^{-c}))$.

Elementary computations then give
\begin{align*}
\sqrt{\hat{\gamma}^2-1}&=\sqrt{\frac{1}{4}\left(\beta+\frac{1}{\beta}\right)^2-1}\\
&=\frac{1}{2}\left|\beta-\frac{1}{\beta}\right|=\frac{1}{2}\frac{1-\beta^2}{\beta},
\end{align*}
\[\frac{\hat{\gamma}}{\sqrt{\hat{\gamma}^2-1}}=\frac{1+\beta^2}{1-\beta^2},\]
so that finally we have
\begin{equation*}
(1-\beta^2)m'_{\mathrm{sc}}(\hat{\gamma})=-2(1-\beta^2)+2(1+\beta^2)= 4\beta^2,
\end{equation*}
and so
\[e^{\frac{t^2(1-\beta_N^2)}{4\beta^2}m'_{\mathrm{sc}}(\hat{\gamma})}=e^{t^2}.\]

\end{proof}
\end{lemma}

\begin{proof}[Proof of Theorem \ref{thm: clt}]
Starting from the representation \eqref{eqn: josephus} and using Proposition \ref{prop: saddle} for the numerator and Proposition \ref{prop: taylor} for the denominator, we have
\begin{align*}
e^{t\langle R_{12}\rangle_{\beta,N}}&=\frac{-\frac{4\pi}{NG''(\gamma)}e^{\frac{t^2(1-\beta_N^2)}{4\beta_N^2}m'_{\mathrm{sc}}(\gamma)} \left(1+O(N^{-c(\tau)})\right)}{-\frac{4\pi}{N G''(\gamma)}(1+O(N^{-c'(\tau)}))}\\
&=e^{\frac{t^2(1-\beta_N^2)}{4\beta_N^2}m'_{\mathrm{sc}}(\gamma)}(1+O(N^{-c})).
\end{align*}
We then conclude using Lemma \ref{lem: variance}. This derivation is valid on the set of overwhelming probability where the statements in Theorem \ref{thm: sc-law} and Theorem \ref{thm: outside-sc} hold.
\end{proof}

\end{document}